\documentclass{article}
\usepackage[utf8]{inputenc}
\usepackage{amsmath}
\usepackage{amsfonts}
\usepackage{amsthm}
\usepackage{amsmath}
\usepackage{amssymb}
\usepackage{amsrefs}
\usepackage[shortlabels]{enumitem}
\usepackage{tikz}
\author{Humara Khan }
	\usepackage{authblk}
\date{June 2022}

\newtheorem{thm}{Theorem}
\newtheorem{cor}[thm]{Corollary}
\newtheorem{lem}[thm]{Lemma}

\newtheorem{defn}[thm]{Definition}

\newtheorem{conj}[thm]{Conjecture}
\newtheorem{claim}[thm]{Claim}

\begin{document}
\author{Gyula Y. Katona}
\author{Humara Khan}
\affil{Department of Computer Science and Information Theory\\
	Budapest University of technology and Economics\\
Budapest, Hungary}

	\title
	{Minimally tough chordal graphs with toughness at most $1/2$}
	\maketitle

\begin{abstract}
	 Let $t$ be a positive real number. A graph is called \emph{$t$-tough} if the removal of any vertex set $S$ that disconnects the graph leaves at most $|S|/t$ components. The toughness of a graph is the largest $t$ for which the graph is $t$-tough. A graph is minimally $t$-tough if the toughness of the graph is $t$ and the deletion of any edge from the graph decreases the toughness. 	A graph is \emph{chordal} if it does not contain an induced cycle of length at least $4$.
	 We characterize the minimally $t$-tough, chordal graphs for all $t\le 1/2$. As a corollary, a characterization of  minimally $t$-tough, interval graphs is obtained for $t\le 1/2$.
\end{abstract}

	\section{Introduction}
	
	All graphs considered in this paper are finite, simple and
	undirected. Let $\omega(G)$ denote the number of components and
	$\alpha(G)$ denote the independence number and $\kappa(G)$ denote the
	connectivity number of the graph $G$. (Using $\omega(G)$ to denote the
	number of components may be confusing, however, most of the literature
	on toughness uses this notation.) For a connected graph $G$, a vertex set $S\subseteq V(G)$ is called a {\em cutset} if $\omega(G-S)>1$.
	
	The notion of toughness was introduced by Chv\'{a}tal in \cite{toughness_intro}.
	
	\begin{defn}
		Let $t$ be a  real number. A graph $G$ is called {\em $t$-tough} if $|S| \ge t \cdot \omega(G-S)$ for any cutset $S \subseteq V(G)$. The {\em toughness} of $G$, denoted by $\tau(G)$, is the largest $t$ for which G is $t$-tough, taking $\tau(K_n) = \infty$ for all $n \ge 1$.
		We say that a cutset $S \subseteq V(G)$ is a {\em tough set} if $\omega(G - S) = |S|/\tau(G)$.
	\end{defn}

Note that a graph is disconnected if and only if its toughness is $0$.
	
	\begin{defn}
		A graph $G$ is said to be {\em minimally $t$-tough} if $\tau(G) = t$ and $\tau(G - e) < t$ for all $e \in E(G)$.
	\end{defn}
	
	It follows directly from the definition that every $t$-tough noncomplete graph is $2t$-connected, implying $\kappa(G) \ge 2 \tau(G)$ for noncomplete graphs. Therefore, the minimum degree of any $t$-tough noncomplete graph is at least $\lceil 2t \rceil$. 
	
	The following conjecture is motivated by a theorem of Mader \cite{ende}, which states that every minimally $k$-connected graph has a vertex of degree $k$.
	
	\begin{conj}[Kriesell \cite{kriesell}] \label{kriesell1}
		Every minimally $1$-tough graph has a vertex of degree $2$.
	\end{conj}
	
	This conjecture can be naturally generalized.
	
	\begin{conj}[Generalized Kriesell's Conjecture] \label{kriesell}
		Every minimally $t$-tough graph has a vertex of degree $\lceil 2t \rceil$.
	\end{conj}

\begin{defn}
	A graph is {\em chordal} if it does not contain an induced cycle of length at least $4$.
\end{defn}

A vertex $v$ of a graph $G$ is {\em simplicial} if its closed neighborhood $N[v]$ forms a clique in $G$.

It was already known that the conjecture is true for chordal graphs when $t \le 1$.

	\begin{thm}[\cite{specgraph}]\label{simplicial}
	Let $t \leq 1$ be a positive rational number. If $t \leq\frac{1}{2}$, then every simplicial vertex of any minimally $t$-tough, chordal graph has degree $1$.
	If $\frac{1}{2}< t \leq 1$, then there exist no minimally $t$-tough, chordal graphs.
\end{thm}

Our goal is to characterize minimally tough, chordal graphs with toughness at most $1/2$. This also gives a new proof
of the Generalized Kriesell's Conjecture for these graphs. In the proof, we will need the notion of clique trees, which we introduce now.

	Let ${\cal K}_G$ be the set of maximal cliques of a graph $G$. A \textit{clique tree} of $G$ is a tree $T$ with vertex set ${\cal K}_G$ such that it satisfies the following \textit{clique-intersection property}: For every pair of distinct cliques $K, K'\in {\cal K}_G$, the set $K\cap K'$ is contained in every clique on the path connecting $K$ and $K'$ in the tree \cite{t4,t5}.  We can also assign weights to the edges of a clique tree. If $K, K'\in {\cal K}_G$ are adjacent in the clique tree, then let $|K\cap K'|$ be the \textit{weight} of this edge. One can show that each weight is at least 1 (see \cite{t4}). 
	
It is also proved in \cite{t4}, that clique trees also satisfy the \textit{induced-subtree property}:  For every vertex $v\in V(G)$ the set of cliques containing $v$ 
induces a subtree of the clique tree $T$.
	\begin{thm}[\cite{t4}] \
			\begin{itemize} 
				\item 
		A connected graph $G$ is chordal if and only if there exists a tree $T=({\cal K}_G, {\cal E}_T)$ for which the clique intersection-property holds. 
		\item A connected graph $G$ is chordal if and only if there exists a tree $T=({\cal K}_G, {\cal E}_T)$ for which the induced-subtree property holds.
			\end{itemize}
	\end{thm}
	In other words, a connected graph $G$ is chordal if and only if it has a clique tree.
	
		\begin{defn}\label{def:TT} We  call a graph a \emph{TT-graph} if it can be obtained 
		 from a tree of maximum degree $\Delta\geq3$ in the following way. Let $Y$ be a subset of vertices satisfying one of the following conditions:		 
		 \begin{enumerate}
		 	\item[(a)] if $\Delta=3$, then $Y$ is the set of all degree 3 vertices so that $Y$ is an independent set and every neighbor of a vertex in $Y$ has degree $2$, or
		 	\item[(b)] if $\Delta\ge 3$, then $Y$ is a subset of some (or all, or none) of the degree 3 vertices so that $Y$ is an independent set and  every neighbor of a vertex in $Y$ has degree $\Delta$.
		 \end{enumerate} 
	 Now remove all vertices of $Y$ (in one step) and for each removed vertex join their three neighbors with a triangle.
	\end{defn}
		\tikzstyle{vertex}=[circle, fill=black,
		minimum size=4pt,inner sep=1pt]
	\tikzstyle{edge} = [draw,very thin,-]
	\tikzstyle{edge} = [draw,ultra thick,-]
	\begin{figure}
		\begin{tikzpicture}[scale=0.42, auto,swap]
			\foreach \pos/ \name in {{(0,3)/a}, {(2.25,1.75)/b}, {(4,1)/c}, {(2.25,3)/h},
				{(0.,1)/d}, {(-1.25,0.25)/e}, {(0,-0.5)/f}, {(6,2)/g}, {(0.5,4.25)/i}, {(-1.25,3.75)/j}, {(4,3)/k},{(3,-0.5)/l},{(8,1)/m},{(8,-0.75)/n},{(8,3)/o},{(6,0.75)/p},{(3,4.5)/q}, {(5,4.5)/r}, {(7,4.5)/s},{(9,4.5)/t},{(9.75,0.5)/u}, {(0.85,1.85)/v},{(4.75,2)/w},{(7.25,2)/x}, {(8,4.5)/y}, {(9,-0.25)/z},{(6,3.25)/a'},{(4,4.5)/b'}, {(2.25,0.5)/c'},{(3.25,2.25)/d'},{(-0.25,4.25)/e'},
				{(-1.25,5)/g'},{(-1,-0.25)/f'},{(-0.25,5.5)/h'},{(3.25,5.5)/i'},{(9.5,5.5)/j'},{(9.75,2.5)/k'},{(11,2.5)/l'},{(8.25,5.5)/m'},{(8.75,1.75)/n'},{(10.25,1.25)/o'},{(9.25,3.5)/p'},{(10.5,4)/q'},{(11,5.5)/r'},{(11.75,4.5)/s'},{(11.5,3.25)/t'}}
			\node[vertex] (\name) at \pos {};
			\tikzstyle{edge} = [draw,ultra thick,-]
			\foreach \source/ \dest in {a/v, b/v, c/w,d/v,o/x,k/w,g/w,g/x,m/x}
			\path[edge] (\source) --  (\dest);
			\tikzstyle{edge} = [draw,very thin,-]
			\foreach \source/ \dest in
			{a/e',b/d',c/c',c/l,c/b,d/f,d/f',e/d,i/a,j/a,g/p,g/a',h/b,m/n,m/z,m/u,k/q,k/r,k/b',o/s,o/t,o/y,t/j',t/m',k'/n',k'/l',k'/o',b'/i',e'/g',e'/h',p'/k',p'/t,p'/q',q'/r',q'/s',q'/t'}
			\path[edge] (\source) --  (\dest);
			\tikzstyle{arrow} = [thin,->,>=stealth]
			\draw[arrow] (12.5,2) -- (14.5,2);
		\end{tikzpicture}
		\begin{tikzpicture}[scale=0.42, auto,swap]
			\foreach \pos/ \name in  {{(0.25,2.75)/a}, {(2.25,1.75)/b}, {(4,1)/c}, {(2.25,3)/h},
				{(0.5,1.15)/d}, {(-0.75,0.25)/e}, {(0.5,-0.5)/f}, {(6,2)/g}, {(0.5,4.25)/i}, {(-1.25,3.75)/j}, {(4,3)/k},{(3,-0.5)/l},{(8,1)/m},{(8,-0.75)/n},{(8,3)/o},{(6,0.75)/p},{(3,4.5)/q}, {(5,4.5)/r}, {(7,4.5)/s},{(9,4.5)/t},{(9.75,0.5)/u}, {(8,4.5)/y}, {(9,-0.25)/z},{(6,3.25)/a'},{(4,4.5)/b'}, {(2.25,0.5)/c'},{(3.25,2.25)/d'},{(-0.25,4.25)/e'},
				{(-1.25,5)/g'},{(-0.5,-0.25)/f'},{(-0.25,5.5)/h'},{(3.25,5.5)/i'},{(9.5,5.5)/j'},{(9.75,2.5)/k'},{(11,2.5)/l'},{(8.25,5.5)/m'},{(8.75,1.75)/n'},{(10.25,1.25)/o'},{(9.25,3.5)/p'},{(10.5,4)/q'},{(11,5.5)/r'},{(11.75,4.5)/s'},{(11.5,3.25)/t'}}
			\node[vertex] (\name) at \pos {};
			\tikzstyle{edge} = [draw,ultra thick,-]
		\foreach \source/ \dest in {a/d, d/b, a/b, k/c,c/g,k/g,g/o,o/m,g/m}
		\path[edge] (\source) --  (\dest);
			\tikzstyle{edge} = [draw, very thin,-]
			\foreach \source/ \dest in { a/e',b/a,b/d',c/c',c/b,c/g,c/k,c/l,d/a,d/b,d/f',e/d,f/d,h/b, i/a, j/a,g/k,g/o,g/p,g/m,g/a',m/n,m/o,m/u,m/z,k/b',k/q,k/r,o/s,o/t,o/y,t/j',t/m',k'/n',k'/l',k'/o',b'/i',e'/g',e'/h',p'/k',p'/t,p'/q',q'/r',q'/s',q'/t'}
			\path[edge] (\source) --  (\dest);
		\end{tikzpicture}
		\caption{Creating a TT-graph from a tree by Definition \ref{def:TT}}
		\label{fig:}
	\end{figure}
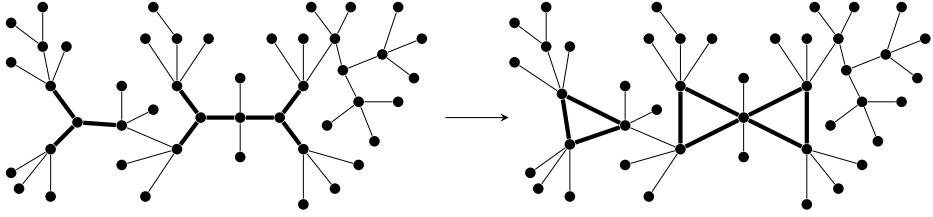
     
     Notice that a TT-graph $G$ may contain triangles, but there is no larger clique in it. Also, any two triangles may only have one common vertex. Trees are also TT-graphs. It is also obvious that TT-graphs are chordal graphs. For TT-graphs we define the \emph{modified degree} for each vertex: $md(v)$ is the number of components in $G-v$. This is clearly the same as $d_G(v)$ minus the number of triangles containing $v$. If a TT-graph $G$ is constructed from a tree $T$, then $md_G(v)=d_T(v)$. Furthermore, it follows from the construction method that all vertices contained in a triangle must have maximal modified degree. Let $\mu(G)$ denote the maximum modified degree in $G$.
	
	During the proof we will use the following general theorem to characterize those graphs that are not minimally $t$-tough.

\begin{thm}[\cite{t3}]\label{th:condition}
	Let $G$ be a connected graph that is not complete and let $t=\tau(G)$.
	Then $G$ is not minimally $t$-tough if and only if $G$ contains an edge $e= uv$ such that the following conditions are met.
	\begin{enumerate}[(a)]		
		\item\label{cond:a} There exist at least $2t + 1$ internally vertex-disjoint $u$-$v$ path in $G$ (including $uv$).
		\item\label{cond:b} Every cutset $S$ in $G$ that is also a $u$-$v$ cutset in $G - e$ satisfies
		$$	|S| \geq (\omega(G - S) + 1)t.$$ 
	\end{enumerate}
\end{thm}

Our main result is  a characterization of chordal graphs with toughness $t\le 1/2$. 

\begin{thm}\label{thm:main}
	For  any positive rational number $t\le 1/2$, a chordal graph is minimally $t$-tough if and only if it is a TT-graph with toughness $t=1/\mu(G)$.
\end{thm}

The proof is divided to several lemmas. Each lemma will cover a subset of the chordal graphs depending on some properties of their clique trees.

\section{Proof}

In the rest of the paper we assume that $t\le 1/2$ is a positive rational number.

\begin{lem}\label{claim1}
	If $G$ is a chordal graph such that in every clique tree, all the edges have weight $1$, then every vertex of $G$ is either a simplicial vertex or  a cut-vertex.
\end{lem}

\begin{proof}
	Notice that the clique intersection property implies that if there are two maximal cliques with intersection of size at least $2$, then the clique tree will have an edge of weight at least $2$ even if the corresponding vertices are not adjacent in the clique tree. This would contradict the assumption of the lemma. 
	
	Clearly, every vertex of the graph must be contained in a maximal clique. If a vertex $v$ is contained in only one, then it is a simplicial vertex. So assume that $v$ is contained in at least two maximal cliques, say $Q_1$ and $Q_2$. The assumption implies that $|Q_1\cap Q_2|= 1$. Now we show that $v$ is a cut-vertex.
	
	Suppose to the contrary that it is not, then for any vertex $x\in Q_1-Q_2$ and any vertex $y\in Q_2-Q_1$ there is a path $P$ connecting $x$ and $y$ and avoiding $v$, choose a shortest such path.   If $xy\in E(G)$, then  the intersection of the maximal clique containing the triangle $vxy$ and $Q_1$  has size at least $2$, a contradiction. If $xy\notin E(G)$,	then the path $P$ with the path $xvy$ forms a cycle of length at least $4$. Since $G$ is chordal, this cycle must be triangulated. Moreover, since we chose a shortest path, $v$ must be connected to all the other vertices on the path $P$. Now if $x'y'$ is the first edge of the path leaving $Q_1$, then the intersection of the maximal clique containing the triangle $vx'y'$ and $Q_1$ has size at least $2$, a contradiction.
\end{proof}


\begin{lem}\label{lem:aa}
	If $G$ is a chordal graph with $\tau(G)=t\le 1/2$ which has a clique tree containing an edge of weight at least $2$, then $G$ is not  minimally $t$-tough.
\end{lem}
\begin{proof}
	Since the toughness of $G$ is $t$ and $t\le 1/2$, it is clear that the graph is not complete, thus we can use Theorem~\ref{th:condition} to prove that the graph is not minimally $t$-tough. By the assumption, a clique tree of $G$ contains an edge that has weight at least $2$. This means that $G$ contains two cliques, $Q_1$ and $Q_2$, that share at least 2 vertices. If there is more than one edge in the clique tree having weight at least $2$, then we have to choose a suitable one. Clearly, removing any edge of the clique tree will separate the clique tree into two components. Pick an edge with weight at least $2$ so that  one of the components  contains only edges of weight 1, and let  $Q_1$ be the clique that corresponds to the vertex in this component. Notice that this implies that if another clique contains a vertex of $Q_1-Q_2$, then it cannot contain any other vertex of $Q_1$. 
	
	Now let $u$ and $v$ be two vertices in $Q_1\cap Q_2$. Since $u$ and $v$ belong to $Q_1$, they are connected by an edge, denote it by $e$. It will be shown that the conditions of Theorem~\ref{th:condition} are satisfied for $e=uv$.
	
	All vertices of the two cliques $Q_1$ and $Q_2$ are adjacent to both $u$ and $v$. Since these cliques are different, there must be a vertex in $Q_1-Q_2$ which is a common neighbor of $u$ and $v$. This gives a path between $u$ and $v$, so together with the $uv$ edge there are two paths between them.
	Since $2t+1\le 2$ follows from our assumption, condition \ref{cond:a}~of Theorem~\ref{th:condition} is satisfied.
	Notice that there must also be a vertex also in $Q_2-Q_1$, this will be used later.

	To show that condition \ref{cond:b} holds as well, consider an arbitrary cutset  $S$  in $G$ that is also a $u$-$v$ cutset in $G-e$. Now that all common neighbors of $u$ and $v$ must belong to $S$. Let $x\in Q_1-Q_2 \subseteq N(u)\cup N(v)$ be an arbitrary element. Thus  $x\in S$ holds. 
	
	We claim that $x$ cannot have a neighbor $y$ outside of $Q_1$ that is connected to $u$ by a path avoiding $x$. Suppose that $yy_1y_2\dots y_m$ is such a path with $y_m=u$. Let $k$ be the smallest index for which $y_k\in Q_1$; such an index exists since $y_m=u\in Q_1$. Since $x\in Q_1$ holds, $yy_1y_2\ldots y_kx$ is a cycle, which must be triangulated because $G$ is chordal. One of these triangles must contain the edge $xy_k$, let the third vertex be $y_{k'}$ with $k'<k$. This vertex is not in $Q_1$ by the choice of $k$. There is a maximal clique containing the triangle $xy_ky_{k'}$, however, this contradicts the choice of $Q_1$.
	
	The above claim implies that either all neighbors of $x$ belong to $Q_1$ or $x$ is a cut-vertex (that separates $u$ from some neighbors of $x$). In the first case $\omega(G-x)t\le \omega(G-x)=1$ holds since $t\le 1/2$.  In the second case, $\omega(G-x)t\le 1$ holds since $G$ is $t$-tough. So in both cases we have $\omega(G-x)\le 1/t$.
	
	Let $S'= S-\{x\}$.
	
	\medskip
	
	\noindent\textit{Case 1}: If $S'$ is a cutset in $G$, then $|S'|\geq \omega(G-S')t$ must hold since $G$ is $t$-tough. Consider the components of $G-S'$. The vertex $x$ is in one of these components, so when we delete $x$ as well, then this component will fall apart into at most $\omega(G-x)$ components. This follows from the observation that  by the above claims the neighborhood of $x$ spans disjoint cliques, so it is impossible to have a path between two such cliques that avoids $x$. Therefore we have
	\[\omega(G-S)\le \omega(G-S')+\omega(G-x)-1\le\omega(G-S')+ \frac{1}{t}-1,\]
	\[\omega(G-S)t+t\le \omega(G-S')t+1.\]
		This implies that
	\[|S|=|S'|+1\geq\omega(G-S')t+1\geq\omega(G-S)t+t=(\omega(G-S)+1)t.\]
	Thus  condition  \ref{cond:b}  of Theorem~\ref{th:condition} is satisfied in this case.
	\medskip
	
	\noindent\textit{Case 2}: If $S-{x}$ is not a cutset, then we prove that $\omega(G-S)\le\omega(G-x)$. It is easy to see that it is implied by the following claim.
	\begin{claim}
		If $S-x$ is not a cutset and two vertices of $G-S$ are in the same component of $G-x$, then they are also in the same component of $G-S$.
	\end{claim}
\begin{proof}
	Let $C_0,C_1,\ldots$ denote the components of $G-x$. Since $u\notin S$, one of the components must contain $u$, let $C_0$ be this component.
We have shown before that the neighbors of $x$ in $C_0$ must belong to $Q_1$. If there are no other neighbors of $x$ then the only component is $C_0$, otherwise there are at least two components. Now let $p$ and $q$ be two vertices of a component of $G-x$ and suppose to the contrary that they belong to different components in $G-S$.

	\vspace{2mm}\noindent\textit{Subcase (i)}: $p,q\in C_0$.
	
Since $S-x$ is not a cutset, this means that there exists a $p$-$q$ path avoiding $S-x$, thus it must contain $x$.  Let $p,\ldots,x^-,x,x^+,\ldots,q$ be such a path.
	Clearly all vertices of this path must be in $C_0$ because of the properties of $x$. This implies that  $x^-,x^+\in Q_1$. Thus $x^-,x^+$ are connected with an edge, so $p,\ldots,x^-,x^+,\ldots,q$ is a $p$-$q$ path that does not contain $x$ and avoids $S-x$, so it avoids $S$, a contradiction.

	\vspace{2mm}\noindent\textit{Subcase (ii)}: $p,q\in C_i$ for some $i>0$.
	
	The choice of $Q_1$ implies that $C_i$ is a chordal graph that satisfies the conditions of Lemma~\ref{claim1}.
	If $p$ and $q$ belong to the same clique, then they belong to the same component of both $G-S$, a contradiction. If they belong to different cliques, then any path connecting them must contain a cut-vertex which belongs to $S-\{x\}$. This implies that $S-\{x\}$ is a cutset, a contradiction.
\end{proof}
	 Now we complete the proof of the lemma. We noticed before that $S$ must contain all common neighbors of $u$ and $v$, and that there exists a common neighbor other than $x$. Thus $S-x$ is not empty. Hence, using  that $\omega(G-x)t\le 1$ and $\omega(G-S)\le\omega(G-x)$, we have
	\[|S|=1+|S-\{x\}|\geq\omega(G-x)t+1\geq\omega(G-S)t+t=(\omega(G-S)+1)t.\]
	
	So  condition  \ref{cond:b} of Theorem~\ref{th:condition} is satisfied in this case, too.
\end{proof}

	 \begin{lem}\label{lem:4} Let $G$ be a chordal graph so that in every clique tree all the edges have weight $1$ and there exist a clique of size at least $4$ in $G$.  This implies that $G$ is not minimally $t$-tough.
	\end{lem}

\begin{proof} 
Let $e=uv$ be an arbitrary edge of a clique $Q$ of size at least $4$. 

Condition  \ref{cond:a} of Theorem~\ref{th:condition} is clearly satisfied since $2t+1\le 2$ and there are two paths of length 2 between $u$ and $v$ through two other vertices of $Q$ (and there is a third path, the $uv$ edge).
To show that condition \ref{cond:b} holds as well, consider an arbitrary cutset  $S$  in $G$ that is also a $u$-$v$ cutset in $G-e$. Clearly, $Q-\{u,v\}\subseteq S$. 

  Let $x$ be an arbitrary vertex in  $Q-\{u,v\}$ and $S'= S-{x}$. Now $S'$ is not empty since $|Q|\ge 4$. Now we can repeat the same argument as in the proof of Lemma \ref{lem:aa}, which implies that Condition  \ref{cond:a} of Theorem~\ref{th:condition} holds in this case, too.
\end{proof}

In the introduction, the modified degree was defined for TT-graphs. However, it can be easily extended for any graph in which each vertex is either simplicial vertex or a cut-vertex: 

\begin{defn}\label{def:md}
	If $G$ is graph in which each vertex is either simplicial vertex or a cut-vertex, then the {\em modified degree} of a vertex $v$, denoted by $md(v)$, is the number of components in $G-v$. Let  $\mu(G)$ denote the maximum modified degree in $G$.
\end{defn}

For these graphs it is easy to determine the toughness.

\begin{lem}\label{lem:mod}
	If $G$ is a connected, noncomplete graph in which each vertex is either a simplicial vertex or a cut-vertex, then $\tau(G)=1/\mu(G)$.
\end{lem}

\begin{proof} If $md(v)=\mu(G)$, then clearly the number of components of $G-v$ is $\mu(G)$, so $\tau(G)\le 1/ \mu(G)$.
	
Let $S$ be a tough set. We claim that $|S|=1$. If  $|S|>1$, then let $x\in S$ be a vertex having the smallest modified degree, and let $S'=S-x$. If $x$ is a simplicial vertex, then $\omega(G-S')=\omega(G-S)$, which contradicts the assumption that $S$ is a tough set. Thus $x$ is a cut-vertex. Consider the components of $G-S'$. The vertex $x$ belongs to one of these components. Since $x$ is a cut-vertex, we have $\omega(G-S)\le \omega(G-S')+\mu(G)-1$, because  deleting $x$ besides $S'$ would give new $\mu(G)$ components instead of this old component. Now
\begin{multline*}
\tau(G)\omega(G-S')\ge \tau(G)\omega(G-S)-\tau(G)(\mu(G)-1)=\\
|S|-\mu(G)\tau(G)+\tau(G)> |S|-1=|S'|,\end{multline*}
since $\mu(G)\tau(G)<1$ by our first claim, and $\tau(G)>0$.  On the other hand, since $\tau(G)$ is the toughness, we have $\omega(G-S')\tau(G)\le |S'|$, a contradiction. 

Hence, it is enough to consider the cutsets of size one to determine the toughness, and it is clear that the most number of components is obtained by deleting a vertex with maximum modified degree.
\end{proof}

 \begin{lem}\label{lem:tt} Let $G$ be a connected, noncomplete chordal graph so that in every clique tree all the edges have weight $1$ and all cliques have   size at most $3$ in $G$.  If $G$ is  minimally $t$-tough, then  it is a TT-graph and $t=1/\mu(G)$.
	\end{lem}

\begin{proof} In Claim~\ref{claim1} we showed that  each vertex of $G$ is either a simplicial vertex or  a cut-vertex, so we can use  Definition~\ref{def:md}.
	
We show that if  $G$ is  minimally $t$-tough and $v$ is contained in a triangle $Q$, then $md(v)=\mu(G)$. Suppose on the contrary that $md(v)\le\mu(G)-1$. Let $e=uw$ be the edge of $Q$ that is not incident to $v$. Consider now $G'=G-e$. If $u$ or $w$ was a cut-vertex in $G$, then the same will hold in $G'$. If either one of them was a simplicial vertex in $G$, then it means that it had degree 2 in $G$. Thus in $G'$ it will have degree 1, implying that it is still a simplicial vertex. So  all vertices of $G'$ are still either simplicial or cut-vertices since for other vertices nothing will change. 

Notice that $md_{G'}(v)=md_{G}(v)+1\le\mu(G)$. Also, $\mu(G')=\mu(G)$, since the modified degree  does not change for any other vertex besides $v$. Thus $\tau(G')=\tau(G)$ holds by Lemma \ref{lem:mod}, implying a contradiction.


When $\mu(G)\ge 3$, then if we replace each triangle with a claw $K_{1,3}$ so that the leaves of the claw correspond to the vertices of the triangle and the center is a news vertex, then we clearly obtain a tree with maximum degree $\mu(G)$. Also, the leaves of each claw have maximum degree in the tree.  So reversing the replacements show that $G$ can be obtained from this tree as described in case (b) of  the definition of TT-graphs. 

In the special case when $\mu(G)=2$,  replacing the triangles with a claw will result in a tree of maximum degree $3$, but all the neighbors of the degree $3$ vertices will have degree $2$. So reversing the replacements show that $G$ can be obtained from this tree as described in case (a) of  the definition of TT-graphs. Therefore we proved that $G$ is a TT-graph with toughness $1/\mu(G)$.
  \end{proof}

\begin{proof}[Proof of Theorem \ref{thm:main}] 
Let $G$ be a minimally $t$-tough, chordal graph with $t\le 1/2$.  Since $G$ is chordal and connected, it has a set of clique trees. By Lemma~\ref{lem:aa}, none of these clique trees contain an edge of weight at least $2$, and by Lemma~\ref{lem:4}, there is no clique of size at least $4$ in $G$. Thus the conditions of Lemma~\ref{lem:tt} are satisfied, so applying the lemma, we conclude that $G$ is a TT-graph.

To complete the proof, we show that any TT-graph $G$ is minimally $t$-tough. It is clear that in a TT-graph every vertex is either a leaf, and therefore simplicial, or a cut-vertex. Thus Lemma~\ref{lem:mod} can be applied, implying that $\tau(G)=1/\mu(G)$. It is also clear that TT-graphs have two kinds of edges: every edge is either a bridge or it is contained in a triangle. 

If an edge $e$ is a bridge, then $G-e$ is disconnected, thus its toughness is $0$. So the toughness of the graph clearly decreases by removing such an edge. 

Now let $e$ be an edge of a triangle $Q$. By the definition of TT-graphs, we know that all vertices of $Q$  have maximum modified degree $\mu(G)$. This implies that in $G-e$, the vertex of $Q$ not incident to $e$  has modified degree $\mu(G)+1$. Since every vertex of $G-e$ is still either simplicial  or  a cut-vertex, Lemma~\ref{lem:mod} shows that the toughness of  $G-e$ is $1/ (\mu(G)+1)< 1/ \mu(G)$. So the toughness of the graph decreases by removing this kind of edges, too.
			\end{proof}

	\section{Implication for interval graphs}
	\begin{defn}
		A graph  is called an interval graph if it is obtained from intervals on the real line in such a way that each interval represents a vertex and the intersection of any two intervals represent an edge. Such graphs are also known as the intersection graph of intervals.
	\end{defn}

An independent triple of vertices $x, y, z$ is called an \emph{asteroidal triple} (AT, for short) if between any pair in the triple there exists a path that avoids the neighborhood of the third vertex.  \emph{Caterpillars} are those trees for which removing the leaves produces a path.

\begin{thm}[\cite{t2}] A graph is an interval graph if and only if it is chordal and asteroidal triple-free.	
\end{thm}	
\begin{cor}
	If $G$ is a minimally $t$-tough, interval graph with $t\le 1/2$, then $G$ is a caterpillar graph.
	\end{cor}
	
	\begin{proof} 
		Since every interval graph is chordal, by Theorem~\ref{thm:main}, $G$ must be a TT-graph. Now we show that $G$ is a tree. Suppose to the contrary that $G$ is not a tree, i.e. it contains a triangle. Each vertex in a triangle must have a distinct neighbor outside the triangle. These neighbors form an AT in the graph, a contradiction.  So $G$  is a tree. However, any tree that is not a caterpillar contains the AT subgraph which is obtained from the claw by subdividing each edge with a vertex. Thus $G$ must be a caterpillar.
		\end{proof}

\section*{Acknowledgment}	The authors are grateful to Kitti Varga for her invaluable help in improving the manuscript.
This research was supported by the Ministry of Innovation and
Technology and the National Research, Development and Innovation
Office within the Artificial Intelligence National Laboratory of Hungary. 	
	
	\bibliographystyle{amsplain}
	
\end{document}